\documentclass[11pt]{article}
\usepackage{moreverb}
\usepackage{amsmath}
\usepackage{amsfonts}
\usepackage{amsthm}
\usepackage{graphicx}
\usepackage{fullpage}

\newtheorem{Theorem}{Theorem}[section]

\newtheorem{Lemma}[Theorem]{Lemma}

\numberwithin{equation}{section}

\newcommand{\ds}{\displaystyle}

\newcommand{\bs}[1]{\boldsymbol{#1}}

\title{Connection coefficients and 
monodromy representations 
for a class of Okubo systems of ordinary 
differential equations, II}
\author{
Shotaro KONNAI \\
{\normalsize Department of Mathematics, Kobe University}}
\date{}
\begin{document}
\maketitle

\begin{abstract}
Explicit connection coefficients and monodromy 
representations are constructed for the canonical solution matrices 
 of Okubo systems ${\rm II}^*$,$ {\rm III}^*$,$ {\rm IV}$ and $ {\rm IV}^*$ of 
ordinary differential equations in Yokoyama's list.  
\end{abstract}

\section{Introduction}
In this paper we construct explicit connection coefficients and monodromy 
representations for the canonical solution matrices 
 of Okubo systems ${\rm II}^*$,$ {\rm III}^*$,$ {\rm IV}$ and $ {\rm IV}^*$ of 
ordinary differential equations in Yokoyama's list. This work is a continuation of our 
previous work \cite{SK}. 

 Recall that Yokoyama \cite{TY} classified the types of 
irreducible rigid Okubo systems 
under the condition that the nontrivial local exponents 
are mutually distinct at each finite singular point; 
this class consists of 
eight types 
${\rm I}$, ${\rm II}$, 
${\rm III}$, ${\rm IV}$ and 
${\rm I}^\ast$, ${\rm II}^\ast$, 
${\rm III}^\ast$, ${\rm IV}^\ast$, 
which is usually referred to as {\em Yokoyama's list}.  
For each type in Yokoyama's list, 
Haraoka constructed 
a canonical form of the Okubo system
\cite{YHEq}, as well as the corresponding 
monodromy representation, 
up to the conjugation by diagonal matrices \cite{YHMon}.  

In \cite{SK}, we determined  explicit monodromy representations
for the canonical solution matrices of types 
${\rm I}$, ${\rm I}^*$, ${\rm II}$ and ${\rm III}$, 
including the diagonal matrix factors
which had not been fixed in \cite{YHMon}.
The argument of \cite{SK} is based on the fact that all the Okubo systems in Yokoyama's list 
can be constructed by finite iterations of the Katz operations of type
\begin{equation}\label{intro:mcadd}
{\rm add}_{(0,\ldots,\rho,\ldots,0)}\circ {\rm mc}_{-\rho-c}\circ 
{\rm add}_{(0,\ldots,c,\ldots,0)}({\bs A})\qquad(c,\rho\in\mathbb{C})
\end{equation}
for Okubo systems.
A characteristic feature of this operation is that for each Okubo system, the resulting system is also an Okubo system, and its explicit form is computed recursively. 
The canonical solution matrix of the resulting Okubo system can also be constructed recursively by the corresponding operation on solution matrices.
By these operations, we specified in \cite{SK} the connection coefficients, as well as the monodromy representations, for the canonical solution matrices of ${\rm I}$, ${\rm I}^*$, ${\rm II}$, ${\rm III}$ in Yokoyama's list.

The purpose of this paper is to give explicit connection coefficients and monodromy representations for the
canonical solution matrices of remaining types ${\rm II}^*$, ${\rm III}^*$, 
${\rm IV}$ and ${\rm IV}^*$.
To derive connection coefficients 
 of these types, we need to carefully  investigate {\em symmetry}  of the systems with respect to the characteristic exponents, in contrast to the cases discussed in \cite{SK} (for the detail, see Section 4).

\section{Review of the method of Katz operations}
We consider the Okubo system 
\begin{equation}\label{eq:Okubo}
(x I_n-T)\frac{d}{dx}Y=AY, \quad T={\rm diag}(t_1 I_{n_1},\ldots,t_r I_{n_r})
\end{equation}
on $\mathcal{D}=\mathbb{C}\backslash \{t_1,\ldots,t_r\}$ 
associated with an $n\times n$ block matrix $ A=(A_{ij})_{i,j=1}^r\in {\rm Mat}(n;\mathbb{C})$ of type $(n_1,\ldots,n_r)$, $n=n_1+\cdots+n_r$,
where $A_{ij}\in {\rm Mat}(n_i,n_j;\mathbb{C})$ $(i,j=1,\ldots,r)$.
This system can be rewritten as the Schlesinger system
\begin{equation}\label{eq:Sch}
\frac{d}{dx}Y=\sum_{k=1}^r\frac{A_k}{x-t_k}Y, \quad A_k=(\delta_{ik}A_{kj})_{i,j=1}^r\quad (k=1\ldots,r).
\end{equation}
We denoting by $\alpha_{k1},\ldots,\alpha_{kn_k}$ 
the eigenvalues of the diagonal block $A_{kk}$ ($k=1,\ldots,r$) 
and by $\rho_1,\ldots,\rho_n$ the eigenvalues of $A=A_1+\cdots+A_r$; 
these eigenvalues are subject to the {\em Fuchs relation}
\begin{equation}
\sum_{k=1}^{r}\sum_{j=k1}^{kn_k}\alpha_j=\sum_{j=1}^{n}\rho_j.
\end{equation}
We assume hereafter that 
\begin{equation}\label{assumedet1}
\alpha_{ki}-\alpha_{kj}\notin 
\mathbb{Z}\backslash \{0\}
\quad(1\le i<j\le n_k),\quad
\alpha_{kj}\notin \mathbb{Z}\quad(1\le j\le n_k)
\end{equation}
for $k=1,\ldots,r$. 

In what follows, we identify \eqref{eq:Sch} with the $r$-tuple of matrices $\bs{A}=(A_1,\ldots,A_r)$.
As for the Katz operations
$({\rm add}_{\bs{a}}$ and ${\rm mc}_{\mu})$,
we follow the notations of our previous paper \cite{SK}. 

\par\medskip
We first recall from \cite{SK} the method of construction for Okubo systems
by Katz operations.

\begin{Theorem}\label{mcadd:eq}{\rm \cite[Section $4$]{SK}}
Let $\rho,c\in \mathbb{C}$ be two complex parameters and choose an index $k$ $(k=1,\ldots,r)$.
 Assume that the Okubo system \eqref{eq:Okubo} satisfies the conditions
${\rm Ker}(A_{kk}+c)=0$ and ${\rm Ker}(A_{kk}-\rho)=0$.
Then the Katz operation
\begin{equation}\label{mcadd}
{\rm add}_{(0,\ldots,\rho,\ldots,0)}\circ {\rm mc}_{-\rho-c}\circ 
{\rm add}_{(0,\ldots,c,\ldots,0)}({\bs A})\qquad(c,\rho\in\mathbb{C})
\end{equation}
transforms the Okubo system \eqref{eq:Okubo} into the Okubo system 
\begin{equation}\label{eq:mcadd}
(x-\widetilde{T})\frac{d}{dx}W=A^{mc}W,\qquad
\quad \widetilde{T}={\rm diag}(t_1I_{n_1},\ldots,t_kI_{\widetilde{n}_k},
\ldots,t_rI_{n_r}),
\end{equation}
of type $(n_1,\ldots,n_{k-1},\widetilde{n}_k,n_{k+1},\ldots,n_r)$,  
$\widetilde{n}_k=
n-{\rm dim Ker}(A-\rho)$, 
where 
\begin{equation}
A^{mc}=\begin{pmatrix}
& A_{1k}(A_{kk}+c)(A_{kk}-\rho)^{-1} &(\rho+c)\xi_{1} &&\\[-3pt]
\ A_{ij}-(\rho+c)\delta_{ij}\ & \vdots & \vdots & A_{ij} \\[-3pt]
& A_{k-1,k}(A_{kk}+c)(A_{kk}-\rho)^{-1} &(\rho+c)\xi_{k-1} &&\\[4pt]
A_{k1}\ \ \ldots\ \ A_{k,k-1}  & A_{kk} & 0 & A_{k,k+1}\ \ \ldots \ \ A_{kr} \\[2pt]
\eta_{1}\ \ \ \ldots\ \ \ \eta_{k-1}  
& 0 & \rho & \eta_{k+1}\ \ \ \ldots\ \ \ \eta_{r} \\[2pt]
& A_{k+1,k}(A_{kk}+c)(A_{kk}-\rho)^{-1} & (\rho+c)\xi_{k+1}& \\[-4pt]
A_{ij} & \vdots & \vdots & A_{ij}-(\rho+c)\delta_{ij}\\[-3pt]
& A_{rk}(A_{kk}+c)(A_{kk}-\rho)^{-1} & (\rho+c)\xi_{r}& 
\end{pmatrix},
\end{equation}
and $\xi_i\in {\rm Mat}(n_i,\widetilde{n}_k)$ and $\eta_j\in {\rm Mat}(\widetilde{n}_k,n_j)$ are matrices characterized by the conditions
\begin{equation}
\xi_i\eta_j=A_{ij}-A_{ik}(A_{kk}-\rho)^{-1}A_{kj}\quad ( i,j=1,\ldots,r;\,i,j\neq k).
\end{equation}
\end{Theorem}

Under the condition \eqref{assumedet1}, for each $j=1,\ldots,r$, the Okubo system \eqref{eq:Okubo} has a unique  $n\times n_j$ solution matrix $\Psi_j(x)$ 
characterized by the local behavior 
\begin{equation}
\Psi_j(x)=F_j(x)(x-t_j)^{A_{jj}},\quad  F_j(x)\in {\rm Mat}(n,n_j;\mathcal{O}_{t_j}),
\quad F_{j}(t_j)=(0,I_{n_j},0)^t,
\end{equation}
around $x=t_j$.
Collecting these solution matrices, we consider the $n\times n$ solution matrix
$\Psi(x)=(\Psi_1(x),\ldots,\Psi_r(x))$.
It is known that this $\Psi(x)$ is a fundamental solution matrix, which we 
call the {\em canonical solution matrix} of \eqref{eq:Okubo} (for the detail, see \cite[Section 1.1]{SK}).


We remark that for each pair $(i,j)$ $ (i,j=1,\ldots,r)$, the solution matrix $\Psi_j(x)$ is expressed uniquely in the form 
\begin{equation}
\Psi_j(x)=\Psi_i(x)C_{ij}+H_{ij}(x)
\end{equation}
around $x=t_i$, where  $C_{ij}$ is an $n_i\times n_j$ constant matrix and $H_{ij}(x)$ is an $n\times n_j$ solution matrix  holomorphic  around $x=t_i$.
These matrices $C_{ij}$ are called the {\em connection coefficients} for the canonical solution matrix $\Psi(x)$.

Fixing a base point $p_0$ on $\mathcal{D}$, we denote by 
$\gamma_i$ ($i=1,\ldots,r$) and $\gamma_\infty$ the homotopy classes 
in $\pi_1(\mathcal{D},p_0)$ of continuous paths 
encircling $x=t_i$ and $x=\infty$ 
in the positive direction, respectively; 
we choose these paths so that $\gamma_\infty \gamma_1\cdots \gamma_r=1$ 
in $\pi_{1}(\mathcal{D},p_0)$ (\cite[Figure 1]{SK}). 
The analytic continuation of $\Psi(x)$ by $\gamma_i$ 
($i=1,\ldots,r$) 
defines an $r$-tuple of matrices 
${\bs M}=(M_1,\ldots,M_r)
\in{\rm GL}(n;\mathbb{C})^r$ as 
\begin{equation}
\gamma_i\!\cdot\!\Psi(x)=\Psi(x)M_i \quad(i=1,\ldots,r). 
\end{equation}
These {\em  monodromy matrices} are expressed as follows in terms of the connection coefficients (\cite[Section 1.2]{SK}):
\begin{equation}\label{eq:Mk}
M_i=
\begin{pmatrix}
1 & & &  &\\[-4pt]
&\ddots & &  &\\
(e(A_{ii})-1)C_{i1}&\dots &e(A_{ii}) & \ldots
 &(e(A_{ii})-1)C_{ir} \\
 & & &\ddots&\\[-4pt]
 & & &  &1
\end{pmatrix}\quad(i=1,\ldots,r),
\end{equation}
where $e(\mu)=e^{2\pi i\mu}$.

In view of Theorem \ref{mcadd}, 
we consider the Okubo system \eqref{eq:mcadd} obtained from \eqref{eq:Okubo} by the Katz 
operation \eqref{mcadd}.
Let $\Psi^{mc}(x)=(\Psi^{mc}_1(x),\ldots,\Psi^{mc}_{k}(x),\ldots,\Psi^{mc}_r(x))$ be the canonical solution matrix
of \eqref{eq:mcadd}, for which the $k$-th component is further decomposed as 
$\Psi^{mc}_{k}(x)=(\Psi^{mc}_{k1}(x),\Psi^{mc}_{k2}(x))$. The components of $\Psi^{mc}(x)$ are characterized by the conditions
\begin{equation}
\Psi^{mc}_j(x)=F^{mc}_j(x)(x-t_j)^{A_{jj}-\rho+c},
\quad F^{mc}_{j}(t_j)=(0,I_{n_j},0)^t,
\quad (j=1,\ldots,r;j\neq k),
\end{equation}
around $x=t_j$, and 
\begin{equation}
\begin{split}
&\Psi^{mc}_{k1}(x)=F^{mc}_{k1}(x)(x-t_k)^{A_{kk}},
\quad F^{mc}_{k1}(t_k)=(0,I_{n_k},0)^t,\\
&
\Psi^{mc}_{k2}(x)=F^{mc}_{k2}(x)(x-t_k)^{\rho},
\quad F^{mc}_{k2}(t_k)=(0,I_{\widetilde{n}_k-n_k},0)^t.
\end{split}
\end{equation}
around $x=t_k$.

For each pair  $(i,j)$ $(i,j=1,\ldots,r)$, the connection matrix $C^{mc}_{ij}$ for $\Psi^{mc}(x)$ is specified by the condition
\begin{equation}
\Psi^{mc}_j(x)=\Psi^{mc}_i(x)C^{mc}_{ij}+H^{mc}_{ij}(x)
\end{equation}
around $x=t_i$.
For each $i=1,\ldots,r$, the monodromy matrix $M_i^{mc}\in {\rm Mat}(\widetilde n;\mathbb{C})$ for $\Psi^{mc}(x)$ with respect to $\gamma_i$ is given by 
\begin{equation}
\begin{split}
M_i^{mc}&=
\begin{pmatrix}
1&&&&\\
&\ddots&&&\\
(e(A_{ii}-\rho-c)-1)C_{i1}^{mc}&\cdots & e(A_{ii}-\rho-c)&
 \cdots & (e(A_{ii}-\rho-c)-1)C_{ir}^{mc}\\
&&&\ddots& \\
&&&&1 
\end{pmatrix}
\quad(i\ne k)
\end{split}
\end{equation}
where $C_{ik}^{mc}=(C_{i(k1)}^{mc},C_{i(k2)}^{mc})$, 
and 
\begin{equation}
\begin{split}
M_k^{mc}&=
\begin{pmatrix}
1&&&&&\\
&\ddots&&&&\\
(e(A_{kk})-1)C_{(k1)1}^{mc}&\cdots&e(A_{kk})&0&\cdots&(e(A_{kk})-1)C_{(k1)r}^{mc}\\
(e(\rho)-1)C_{(k2)1}^{mc}&\cdots &0&e(\rho)&\cdots
 &(e(\rho)-1)C_{(k2)r}^{mc}\\
&&&&\ddots&\\
&&&&&1
\end{pmatrix}.
\end{split}
\end{equation}
where $C_{kj}^{mc}=(C_{(k1)j}^{mc},C_{(k2)j}^{mc})^t$.
\par\medskip

\begin{Theorem}\label{conn:mcadd}{\rm \cite[Theorem $4.4$]{SK}}
The connection coefficients $C_{ij}^{mc}$
and $C_{i(k1)}^{mc}$, $C_{(k1)j}^{mc}$ 
$(i,j\ne k)$ 
for the canonical 
solution matrix $\Psi^{mc}(x)$ are 
determined from the connection coefficients $C_{ij}$ 
for $\Psi(x)$ by  the following recurrence formulas :
\begin{equation}\label{cij}
\begin{split}
C_{ij}^{mc}
&=
\begin{cases}
\displaystyle
\left(\frac{t_i-t_k}{t_j-t_k}\right)^{\rho+c}
\frac{e(\frac{-1}{2}(\rho+c))\Gamma(\rho+c-A_{ii})}{\Gamma(-A_{ii})}
C_{ij}\frac{\Gamma(A_{jj}-\rho-c+1)}{\Gamma(A_{jj}+1)}
\ &(j<i;\,i,j\neq k)\\[10pt]
\displaystyle
\left(\frac{t_i-t_k}{t_j-t_k}\right)^{\rho+c}
\frac{e(\frac{1}{2}(\rho+c))\Gamma(\rho+c-A_{ii})}{\Gamma(-A_{ii})}
C_{ij}\frac{\Gamma(A_{jj}-\rho-c+1)}{\Gamma(A_{jj}+1)}
\ &(i<j;\,i,j\neq k),
\end{cases}
\end{split}
\end{equation}
where $\delta(i<j)=1$ if $i<j$ and $\delta(i<j)=0$ otherwise. 
Similarly, 
\begin{equation}\label{Ci(k1)}
\begin{split}
C_{i(k1)}^{mc}&=
\begin{cases}
\displaystyle
(t_i-t_k)^{\rho+c}e(\tfrac{1}{2}(\rho+c))
\frac{\Gamma(\rho+c-A_{ii})}{\Gamma(-A_{ii})}
C_{ik}\frac{\Gamma(A_{kk}-\rho)}{\Gamma(A_{kk}+c)}
\quad &(i<k)\\[10pt]
\displaystyle
(t_i-t_k)^{\rho+c}e(\tfrac{-1}{2}(\rho+c))
\frac{\Gamma(\rho+c-A_{ii})}{\Gamma(-A_{ii})}
C_{ik}\frac{\Gamma(A_{kk}-\rho)}{\Gamma(A_{kk}+c)}
\quad &(k<i)
\end{cases}
\end{split},
\end{equation}
\begin{equation}\label{C(k1)j}
\begin{split}
C_{(k1)j}^{mc}
&=
\begin{cases}
\displaystyle
-e(\tfrac{-1}{2}(\rho+c))(t_j-t_k)^{-\rho-c}
\frac{\Gamma(1+\rho-A_{kk})}{\Gamma(1-A_{kk}-c)}C_{kj}
\frac{\Gamma(A_{jj}-\rho-c+1)}{\Gamma(A_{jj}+1)}
\quad &(j<k)\\[10pt]
\displaystyle
-e(\tfrac{1}{2}(\rho+c))(t_j-t_k)^{-\rho-c}
\frac{\Gamma(1+\rho-A_{kk})}{\Gamma(1-A_{kk}-c)}C_{kj}
\frac{\Gamma(A_{jj}-\rho-c+1)}{\Gamma(A_{jj}+1)}
\quad &(k<j).
\end{cases}
\end{split}
\end{equation}
\end{Theorem} 
We remark that the connection coefficients $C_{i(k2)}$ and $C_{(k2)j}$ are not specified in this theorem. In the cases of the Okubo systems in Yokoyama's list, however,  those coefficients can be determined by the symmetry of the characteristic exponents (Section 4).

\section{Canonical forms of Okubo system of types ${\rm II}^*$,${\rm III}^*$,${\rm IV}$,${\rm IV}^*$}

Yokoyama's list is a class of rigid irreducible Okubo systems 
such that 
each diagonal block
$A_{ii}\in{\rm Mat}(n_i;\mathbb{C})$ $(i=1,\ldots,r)$ of $A$ 
has $n_i$ mutually distinct eigenvalues 
and that the matrix 
$A\in{\rm Mat}(n;\mathbb{C})$ is diagonalizable. 
In this class we can assume that 
the diagonal blocks $A_{ii}$ ($i=1,\ldots,r$) 
are diagonal matrices with mutually distinct entries. 
We assume that $A$ is diagonalized as
\begin{equation}
A\sim{\rm diag}(\rho_1I_{m_1},\ldots,\rho_{q}I_{m_q})\quad (m_1\geq\cdots \geq m_q)
\end{equation} 
with mutually distinct $\rho_1,\ldots,\rho_q\in\mathbb{C}$. 
Then the 
Okubo systems in Yokoyama's list 
are characterized  
by the following eight pairs $(n_1,\ldots,n_r)$, 
$(m_1,\ldots,m_q)$ of partitions of $n$ respectively.
\begin{equation}
\begin{array}{clcl}
({\rm I})_n:
&
\begin{cases}
(n_1,n_2)=(n-1,1)\\
(m_1,\ldots,m_n)=(1,\ldots,1)
\end{cases}\quad 
&
({\rm I}^*)_n:
&
\begin{cases}
(n_1,\ldots,n_n)=(1,\ldots,1)\\
(m_1,m_2)=(n-1,1)
\end{cases}\quad 
\\[16pt]
({\rm II})_{2n}:
&
\begin{cases}
(n_1,n_2)=(n,n)\\
(m_1,m_2,m_3)=(n,n-1,1)
\end{cases}\quad 
&
({\rm II^\ast})_{2n}:
&
\begin{cases}
(n_1,n_2,n_3)=(n,n-1,1)\\
(m_1,m_2)=(n,n)
\end{cases}\quad 
\\[16pt]
({\rm III})_{2n+1}:
&
\begin{cases}
(n_1,n_2)=(n+1,n)\\
(m_1,m_2,m_3)=(n,n,1)
\end{cases}
&
({\rm III^\ast})_{2n+1}:
&
\begin{cases}
(n_1,n_2,n_3)=(n,n,1)\\
(m_1,m_2)=(n+1,n)
\end{cases}
\\[16pt]
({\rm IV})_{6}:
&
\begin{cases}
(n_1,n_2)=(4,2)\\
(m_1,m_2,m_3)=(2,2,2)
\end{cases}\quad 
&
({\rm IV^\ast})_{6}:
&
\begin{cases}
(n_1,n_2,n_3)=(2,2,2)\\
(m_1,m_2)=(4,2)
\end{cases} 
\end{array}
\end{equation}
In what follows, we use the symbols $({\rm II}^*)_{2n}$, $({\rm III}^*)_{2n+1}$,$ ({\rm IV})_6$, and $({\rm IV}^*)_6$ to refer to the corresponding tuples $\bs{A}$. As for the nontrivial eigenvalues of $A_{11}$, $A_{22}$, $A_{33}$ and $A$, we use the notations 
$\alpha_i=\alpha_i^{(l)}$, $\beta_i=\beta_i^{(l)}$, $\gamma_i=\gamma_i^{(l)}$ and $\rho_i=\rho_i^{(l)}$, where $l$ denotes the rank $l$ of the Okubo system. In the following  theorems,
we give canonical forms of the Okubo systems of Yokoyama's list  which we will use throughout 
this paper.
\begin{Theorem}
The systems $({\rm II}^*)_{2n}$ and $({\rm III}^*)_{2n+1}$ are expressed in the form 
\begin{equation}
(x-T)\frac{d}{dx}Y=
\begin{pmatrix}
\alpha&A_{12}&A_{13}\\
A_{21}&\beta&A_{23}\\
A_{31}&A_{32}&\gamma
\end{pmatrix}Y,
\end{equation}
where the matrices $A_{kl}$ are given as follows.\\
$({\rm II}^*)_{2n}$
\begin{equation}\label{eq:II*}
\begin{split}
(A_{12})_{ij}&=\frac{\prod_{k\neq i}^n(\beta_j+\alpha_k-\rho_1-\rho_2)}{\prod_{k\neq j}^{n-1}(\beta_j-\beta_k)}\quad (i=1,\ldots,n;j=1,\ldots,n-1),\\
(A_{13})_i&=1\quad (i=1,\ldots,n),\\
(A_{23})_i&=1\quad  (i=1,\ldots,n-1),\\
(A_{21})_{ij}&=(\alpha_j-\rho_1)(\alpha_j-\rho_2)\frac{\prod_{k\neq i}^{n-1}(\alpha_j+\beta_k-\rho_1-\rho_2)}{\prod_{k\neq j}^n(\alpha_j-\alpha_k)}\quad 
 (i=1,\ldots,n-1,j=1,\ldots,n),\\
(A_{31})_{j}&=-(\alpha_j-\rho_1)(\alpha_j-\rho_2)\frac{\prod_{k=1}^{n-1}(\alpha_j+\beta_k-\rho_1-\rho_2)}{\prod_{k\neq j}^n(\alpha_j-\alpha_k)}\quad
(j=1,\ldots,n),\\
(A_{32})_{j}&=-\frac{\prod_{k=1}^n(\beta_j+\alpha_k-\rho_1-\rho_2)}{\prod_{k\neq j}^{n-1}(\beta_j-\beta_k)}\quad 
(j=1;\ldots,n-1),\\
\end{split}
\end{equation}
$({\rm III}^*)_{2n+1}$
\begin{equation}\label{eq:III*}
\begin{split}
(A_{12})_{ij}&=(\beta_j-\rho_1)\frac{\prod_{k\neq i}^n(\beta_j+\alpha_k-\rho_1-\rho_2)}{\prod_{k\neq j}^{n}(\beta_j-\beta_k)}\quad
 (i,j=1,\ldots,n),\quad \hspace{110pt}\\
(A_{13})_i&=1\quad (i=1,\ldots,n),\\
(A_{23})_i&=1\quad (i=1,\ldots,n),\\
(A_{21})_{ij}&=(\alpha_j-\rho_1)\frac{\prod_{k\neq i}^{n}(\alpha_j+\beta_k-\rho_1-\rho_2)}{\prod_{k\neq j}^n(\alpha_j-\alpha_k)}\quad
 (i,j=1,\ldots,n),\\
(A_{31})_{j}&=-(\alpha_j-\rho_1)\frac{\prod_{k=1}^{n}(\alpha_j+\beta_k-\rho_1-\rho_2)}{\prod_{k\neq j}^n(\alpha_j-\alpha_k)}
\quad (j=1,\ldots,n),\\
(A_{32})_{j}&=-(\beta_j-\rho_1)\frac{\prod_{k=1}^n(\beta_j+\alpha_k-\rho_1-\rho_2)}{\prod_{k\neq j}^{n}(\beta_j-\beta_k)}\quad
(j=1,\ldots,n).\\
\end{split}
\end{equation}
\end{Theorem}
\noindent The symbol $\prod_{k \neq i}^m$ stands for the product over $k$ such that $1\leq k\leq m$ and $k\neq i$.

\begin{Theorem}
The system $({\rm IV})_{6}$ is expressed in the form 
\begin{equation}
(x-T)\frac{d}{dx}Y=
\begin{pmatrix}
\alpha&A_{12}\\
A_{21}&\beta\\
\end{pmatrix}Y
\end{equation}
where the matrices $A_{kl}$ are given as follows.
\begin{equation}\label{eq:IV}
\begin{split}
(A_{12})_{ij}&=\frac{\prod_{k\neq i}^3(\alpha_k+\alpha_4+\beta_j-\rho_1-\rho_2-\rho_3)}{\prod_{k\neq j}^{2}(\beta_j-\beta_k)}\quad (i=1,2,3;j=1,2),\\
(A_{12})_{4j}&=\frac{1}{\prod_{k\neq j}^{2}(\beta_j-\beta_k)}\quad (j=1,2),\\
(A_{21})_{ij}&=\prod_{k=1}^3(\alpha_j-\rho_k)
\frac{\prod_{k\neq i}^2(\alpha_j+\alpha_4+\beta_k-\rho_1-\rho_2-\rho_3)}
{\prod_{k\neq j}^4(\alpha_j-\alpha_k)}\quad (i=1,2;j=1,2,3),\\
(A_{21})_{i4}&=\prod_{k=1}^3(\alpha_j-\rho_k)
\frac{\prod_{k=1}^{3}\prod_{l\neq i}^2(\alpha_k+\alpha_4+\beta_{l}-\rho_1-\rho_2-\rho_3)}
{\prod_{k\neq j}^4(\alpha_j-\alpha_k)}\quad (i=1,2).\\
\end{split}
\end{equation}
\end{Theorem}

\begin{Theorem}
The system $({\rm IV}^*)_{6}$ is expressed in the form 
\begin{equation}\label{eq:IV*}
(x-T)\frac{d}{dx}Y
=\begin{pmatrix}
\alpha_1&0&h_{212}&h_{222}&1&h_{212}h_{222}\\
0&\alpha_2&h_{112}&h_{122}&1&h_{112}h_{122}\\
h_{122}&h_{222}&\beta_1&0&1&-h_{122}h_{222}\\
h_{112}&h_{212}&0&\beta_2&1&-h_{112}h_{212}\\
-h_{112}h_{122}&-h_{212}h_{222}&-h_{112}h_{212}&-h_{122}h_{222}&\gamma_1&0\\
-1&-1&1&1&0&\gamma_2
\end{pmatrix}DY,
\end{equation}
where the notations 
\begin{equation}
h_{ijk}=\alpha_i+\beta_j+\gamma_k-2\rho_1-\rho_2,
\end{equation}
and $D={\rm diag}(a_1,a_2,b_1,b_2,c_1,c_2)$ denotes the diagonal matrix defined by
\begin{equation}
a_j=\frac{\alpha_j-\rho_1}{\prod_{k\neq j}^2(\alpha_j-\alpha_{k})},\quad bj=\frac{\beta_j-\rho_1}{\prod_{k\neq j}^2(b_j-b_{k})},
\quad c_j=\frac{\gamma_j-\rho_1}{\prod_{k\neq j}^2(\gamma_j-\gamma_{k})}\quad (j=1,2).
\end{equation}
\end{Theorem}
\begin{proof}[Proof of Theorems 3.1, 3.2 and 3.3] 
{\bf Case} ${\rm II}^*$,${\rm III}^*$:
 Firstly, the system $({\rm III}^*)_3$ is obtained by ${\rm mc}_{\mu_1}$ from the differential equation
\begin{equation}
\frac{dy}{dx}=\left(\frac{\alpha_1^{(1)}}{x-t_1}+\frac{\beta_1^{(1)}}{x-t_2}+
\frac{\gamma_1^{(1)}}{x-t_3}\right)y.
\end{equation}
The resulting system is given by
\begin{equation}
(x-T)\frac{d}{dx}Y
=\begin{pmatrix}
\alpha_1 &\beta_1-\rho_1 &\rho_1+\rho_2-\alpha_1-\beta_1\\
\alpha_1-\rho_1&\beta_1&\rho_1+\rho_2-\alpha_1-\beta_1 \\
\alpha_1-\rho_1&\beta_1-\rho_1&\gamma
\end{pmatrix}Y
\end{equation}
with characteristic exponents specified by
\begin{equation}
\begin{cases}
\alpha_1=\alpha_1^{(1)}+\mu_1,\quad \beta_1=\beta_1^{(1)}+\mu_1,\quad \gamma=\gamma_1^{(1)}+\mu_1,\\
\rho_1=\mu_1,\quad \rho_2=\alpha_1^{(1)}+\beta_1^{(1)}+\gamma_1^{(1)}+\mu_1.
\end{cases}
\end{equation}
It can be transformed into the canonical form
\begin{equation}
({\rm III}^*)_3=
\begin{pmatrix}
\alpha_1 &\beta_1-\rho_1 &1\\
\alpha_1-\rho_1&\beta_1&1 \\
\frac{\alpha_1-\rho_1}{\rho_1+\rho_2-\alpha_1-\beta_1}&
\frac{\beta_1-\rho_1}{\rho_1+\rho_2-\alpha_1-\beta_1}&\gamma
\end{pmatrix}
\end{equation}
 by conjugation with the diagonal matrix
\begin{equation}
{\rm diag}(1,\,1,\,\rho_1+\rho_2-\alpha_1-\beta_1).
\end{equation}

Starting from the $({\rm III}^*)_3$, the Okubo systems $({\rm II}^*)_{2n}$ and $({\rm III}^*)_{2n+1}$ can be constructed inductively by the following Katz operations:
\begin{equation}\label{mcadd23*}
\begin{split}
({\rm II}^*)_{2n}&={\rm add}_{(\rho_2,0,0)}\circ{\rm mc}_{-a_{n-1}-\rho_2}\circ{\rm add}_{(a_{n-1},0,0)}({\rm III}^*)_{2n-1},\\
({\rm III}^*)_{2n+1}&={\rm add}_{(0,\rho_2,0)}\circ{\rm mc}_{-b_{n}-\rho_2}\circ{\rm add}_{(0,b_{n},0)}({\rm II}^*)_{2n}.
\end{split}
\end{equation}
Therefore, by Theorem \ref{mcadd:eq} we obtain canonical forms of the Okubo system types $({\rm II}^*)_{2n}$ and $({\rm III}^*)_{2n+1}$ inductively.
We remark that the construction of the system is equivalent to the computation of $\xi$ and  $\eta$ of Theorem \ref{mcadd:eq}.

Assume that the system $({\rm III}^*)_{2n-1}$ is given as in our canonical form of Theorem 3.1.
By Theorem \ref{mcadd:eq}, the Katz operation \eqref{mcadd23*}
for $({\rm III^*})_{2n-1}$ gives rise to the Okubo system $(x-T)\frac{d}{dx}W = A^{mc}W$ 
where
\begin{equation}\label{II*}
A^{mc}=\begin{pmatrix}
\alpha&0&A_{12}&A_{13}\\
0&\rho_2&\eta_2&\eta_3\\
A_{21}(\alpha+a_{n-1})(\alpha-\rho_2)^{-1}&(\rho_2+a_{n-1})\xi_2&\beta-a_{n-1}-\rho_2&A_{23}\\
A_{31}(\alpha+a_{n-1})(\alpha-\rho_2)^{-1}&(\rho_2+a_{n-1})\xi_3&A_{32}&\gamma-a_{n-1}-\rho_2
\end{pmatrix}.
\end{equation}
Renaming the characteristic exponents of \eqref{II*} as
\begin{equation}
\begin{cases}
\alpha_i^{(2n)}=\alpha_i\quad(1\leq i \leq n-1),\quad 
\alpha_{n}^{(2n)}=\rho_2,\\
\beta_i^{(2n)}=\beta_i-\rho_2-a_{n-1}\quad(1\leq i\leq n),\\
\gamma^{(2n)}=\gamma-a_{n-1}-\rho_2,\quad
\rho_1^{(2n)}=-a_{n-1}, \quad \rho_2^{(2n)}=\rho_1,
\end{cases}
\end{equation}
we see that $A^{mc}$ is in the canonical form of type $({\rm II}^*)_{2n}$ except for the vectors $\xi_k,\eta_l$ to be determined by the formula
\begin{equation}
\xi_k\eta_l=A_{kl}-A_{k1}(A_{11}-\rho_2)A_{1l}.
\end{equation}
The values of $(\xi_k)_i(\eta_l)_j$ are computed as follows by partial fractional expansions:
\begin{equation}
\begin{split}
(\xi_2)_i(\eta_2)_j&=(\beta_i-\rho_2)\delta_{ij}-\sum_{p=1}^{n-1}
\frac{\alpha_p-\rho_1}{\alpha_p-\rho_2}\frac{\prod_{k\neq i}^{n-1}(\alpha_p+\beta_k-\rho_1-\rho_2)}{\prod_{k\neq p}^{n-1}(\alpha_p-\alpha_k)}
(\beta_j-\rho_1)\frac{\prod_{k\neq p}^{n-1}(\beta_j+\alpha_k-\rho_1-\rho_2)}{\prod_{k\neq j}^{n-1}(\beta_j-\beta_k)}\\
&=(\beta_i-\rho_2)\delta_{ij}-\frac{(\beta_j-\rho_1)}{\prod_{k\neq j}^{n-1}(\beta_j-\beta_k)}\sum_{p=1}^{n-1}
\frac{\alpha_p-\rho_1}{\alpha_p-\rho_2}\frac{\prod_{k\neq i}^{n-1}(\alpha_p+\beta_k-\rho_1-\rho_2)}{\prod_{k\neq p}^{n-1}(\alpha_p-\alpha_k)}
\prod_{k\neq p}^{n-1}(\beta_j+\alpha_k-\rho_1-\rho_2)\\
&=(\beta_i-\rho_2)\delta_{ij}-\frac{(\beta_j-\rho_1)\prod_{k=1}^{n-1}(\beta_j+\alpha_k-\rho_1-\rho_2)}{\prod_{k\neq j}^{n-1}(\beta_j-\beta_k)}
\sum_{p=1}^{n-1}
\frac{\alpha_p-\rho_1}{\alpha_p-\rho_2}\frac{\prod_{k\neq i,j}^{n-1}(\alpha_p+\beta_k-\rho_1-\rho_2)}{\prod_{k\neq p}^{n-1}(\alpha_p-\alpha_k)}
\\
&=(\beta_i-\rho_2)\delta_{ij}-\frac{(\beta_j-\rho_1)\prod_{k=1}^{n-1}(\beta_j+\alpha_k-\rho_1-\rho_2)}{\prod_{k\neq j}^{n-1}(\beta_j-\beta_k)}
\sum_{p=1}^{n-1}
\frac{\prod_{k\neq i,j}^{n}(\alpha_p'+\beta_k'-\rho_1'-\rho_2')}{\prod_{k\neq p}^{n}(\alpha_p'-\alpha_k')}
\\
&=\frac{(\beta_j-\rho_1)\prod_{k=1}^{n-1}(\beta_j+\alpha_k-\rho_1-\rho_2)}{\prod_{k\neq j}^{n-1}(\beta_j-\beta_k)}
\frac{\prod_{k\neq i,j}^{n}(\alpha_n'+\beta_k'-\rho_1'-\rho_2')}{\prod_{k\neq n}^{n}(\alpha_n'-\alpha_k')}
\\
&=(\rho_2-\rho_1)\frac{(\beta_j-\rho_1)\prod_{k=1}^{n-1}(\beta_j+\alpha_k-\rho_1-\rho_2)}{\prod_{k\neq j}^{n-1}(\beta_j-\beta_k)}
\frac{\prod_{k\neq i,j}^{n-1}(\beta_k-\rho_1)}{\prod_{k\neq n}^{n}(\rho_2-\alpha_k)}
\end{split}
\end{equation}
where $\alpha_i=\alpha_i'$, $\beta_i=\beta_i'+\alpha_{n}'-\beta_{n}'$,  $\rho_1=\rho_1'+\rho_2'-\beta_{n}'$, $\rho_2=\alpha_{n}'$, and
\begin{equation}
\xi_3\eta_3=\gamma-\rho_2-\sum_{p=1}^{n-1}
\frac{\alpha_p-\rho_1}{\alpha_p-\rho_2}\frac{\prod_{k=1}^{n-1}(\alpha_p+\beta_k-\rho_1-\rho_2)}{\prod_{k\neq p}^{n-1}(\alpha_p-\alpha_k)}
=\frac{\prod_{k=1}^{n-1}(\beta_k-\rho_1)}{\prod_{k=1}^{n-1}(\rho_2-\alpha_k)}.
\end{equation}
Therefore we can choose $(\xi_2)_i,(\eta_2)_j,\xi_3,\eta_3$ as 
\begin{equation}
\begin{split}
(\xi_2)_i&=(\rho_2-\rho_1)\frac{\prod_{k\neq i}^{n-1}(\beta_k-\rho_1)}{\prod_{k=1}^{n-1}(\rho_2-\alpha_k)},
\quad (\eta_2)_j=\frac{\prod_{k=1}^{n-1}(\beta_j+\alpha_k-\rho_1-\rho_2)}{\prod_{k\neq j}^{n-1}(\beta_j-\beta_k)},\\
 \xi_3&=\frac{\prod_{k=1}^{n-1}(\beta_k-\rho_1)}{\prod_{k=1}^{n-1}(\rho_2-\alpha_k)},
\quad \eta_3=1,
\end{split}
\end{equation}
so that $A^{mc}$ gives rise to the canonical form of type $({\rm II}^*)_{2n}$. Similarly, applying the Katz operation to the canonical form of $({\rm II}^*)_{2n}$, we obtain the canonical form of type $({\rm III}^*)_{2n+1}$.
\par\medskip

\noindent
{\bf Case} ${\rm IV}$:
The Okubo system of type  $({\rm IV})_6$ is constructed by
\begin{equation}
({\rm IV})_6={\rm add}_{(\rho_3,0)}\circ{\rm mc}_{-c-\rho_3}\circ{\rm add}_{(c,0)}({\rm III})_5.
\end{equation}
One can directly verify that the the canonical form of type $({\rm IV})_6$ is 
obtained from the canonical form of type $({\rm III})_5$ in \cite{SK} by choosing $\xi$ and $\eta$ appropriately.
\par\medskip

\noindent
{\bf Case} ${\rm IV^*}$:
The Okubo system of type  $({\rm IV}^*)_6$ is constructed by
\begin{equation}
({\rm IV}^*)_6={\rm add}_{(0,0,\rho_1)}\circ{\rm mc}_{-c-\rho_1}\circ{\rm add}_{(0,0,c)}({\rm III}^*)_5.
\end{equation}
One can directly verify that the the canonical form of type $({\rm IV}^*)_6$ is 
obtained from the canonical form of type $({\rm III}^*)_5$ of Theorem 3.1 by choosing $\xi$ and $\eta$ appropriately.
\end{proof}

\section{Construction of connection coefficients of Okubo system}
\subsection{Main results}
In this section we present explicit formulas for the connection matrices $C_{kl}$ of types ${\rm II}^*$,${\rm III}^*$,${\rm IV}$,${\rm IV}^*$ corresponding to the canonical forms of  Theorems 3.1, 3.2, 3.3.
\par\medskip

\begin{Theorem}
For the 
Okubo systems \eqref{eq:II*} and \eqref{eq:III*} of types  $({\rm II}^*)_{2n}$ and $({\rm III}^*)_{2n+1}$, 
the connection coefficients of the canonical solution matrices 
are expressed as follows:
\par\medskip
\noindent
$({\rm II}^*)_{2n}$
\begin{equation}
\begin{split}
(C_{12}^{(2n)})_{ij}
&=\displaystyle (-1)^{n-1}
e(\tfrac{1}{2}(\rho_1+\rho_2-\alpha_i-\beta_j-\gamma)
\frac{(t_1-t_2)^{\rho_1+\rho_2-\alpha_i-\gamma}}
{(t_2-t_1)^{\rho_1+\rho_2-\beta_j-\gamma}}
\left(\frac{t_1-t_3}{t_2-t_3}\right)^{\rho_1+\rho_2-\alpha_i-\beta_j}\\
&\frac{\Gamma(-\alpha_i)\Gamma(\beta_j+1)}{\Gamma(1+\rho_1-\alpha_i)\Gamma(1+\rho_2-\alpha_i)}
\frac{\prod_{k\neq j}^{n-1}\Gamma(\beta_j-\beta_k)}
{\prod_{k\neq i}^n\Gamma(\beta_j+\alpha_{k}-\rho_1-\rho_2)}
\frac{\prod_{k\neq i}^n\Gamma(1+\alpha_k-\alpha_i)}
{\prod_{k\neq j}^{n-1}\Gamma(1+\rho_1+\rho_2-\alpha_i-\beta_k)}
\end{split}
\end{equation}

\begin{equation}
\begin{split}
(C_{13}^{(2n)})_{i}
&=(-1)^{n}e(\tfrac{1}{2}(\rho_1+\rho_2-\alpha_i-\beta_1-\gamma)
\frac{(t_1-t_3)^{\rho_1+\rho_2-\alpha_i-\beta_1}}
{(t_3-t_1)^{\rho_1+\rho_2-\beta_1-\gamma}}
\left(\frac{t_1-t_2}{t_3-t_2}\right)^{\rho_1+\rho_2-\alpha_i-\gamma}\\
&(\rho_1+\rho_2-\alpha_i-\beta_1)^{-1}
\frac{\Gamma(\gamma+1)\Gamma(-\alpha_i)}
{\Gamma(1+\rho_1-\alpha_i)\Gamma(1+\rho_2-\alpha_i)}
\frac{\prod_{k\neq i}^n\Gamma(1+\alpha_k-\alpha_i)}{\prod_{k=1}^{n-1}\Gamma(1+\rho_1+\rho_2-\alpha_i-\beta_k)}
\end{split}
\end{equation}

\begin{equation}
\begin{split}
(C_{23}^{(2n)})_i
&=(-1)^{n}
e(\tfrac{1}{2}(\rho_1+\rho_2-\alpha_1-\beta_i-\gamma))
\frac{(t_2-t_3)^{\rho_1+\rho_2-\alpha_1-\beta_i}}
{(t_3-t_2)^{\rho_1+\rho_2-\alpha_1-\gamma}}
\left(\frac{t_2-t_1}{t_3-t_1}\right)^{\rho_1+\rho_2-\beta_i-\gamma}\\
&(\rho_1+\rho_2-\alpha_1-\beta_i)^{-1}
\Gamma(\gamma+1)\Gamma(-\beta_i)
\frac{\prod_{k\neq i}^{n-1}\Gamma(1+\beta_k-\beta_i)}
{\prod_{k=1}^n\Gamma(1+\rho_1+\rho_2-\alpha_k-\beta_i)}
\end{split}
\end{equation}

\begin{equation}
\begin{split}
(C_{21}^{(2n)})_{ij}
&=\displaystyle (-1)^{n}
e(\tfrac{-1}{2}(\alpha_j+\beta_i+\gamma-\rho_1-\rho_2)
\frac{(t_2-t_1)^{\rho_1+\rho_2-\beta_i-\gamma}}
{(t_1-t_2)^{\rho_1+\rho_2-\alpha_j-\gamma}}
\left(\frac{t_2-t_3}
{t_1-t_3}\right)^{\rho_1+\rho_2-\alpha_i-\beta_j}\\
&\frac{\Gamma(\alpha_j+1)\Gamma(-\beta_i)}{\Gamma(1+\rho_1-\beta_j)
\Gamma(1+\rho_2-\beta_j)}
\frac{\prod_{k\neq i}^{n-1}\Gamma(\beta_k-\beta_i+1)}
{\prod_{k\neq j}^n\Gamma(1+\rho_1+\rho_2-\beta_i-\alpha_{k})}
\frac{\prod_{k\neq j}^n\Gamma(\alpha_j-\alpha_k))}
{\prod_{k\neq i}^{n-1}\Gamma(\alpha_j+\beta_k-\rho_1-\rho_2)}
\end{split}
\end{equation}

\begin{equation}
\begin{split}
(C_{31}^{(2n)})_{j}
&=(-1)^{n}e(\tfrac{1}{2}(\rho_1+\rho_2-\alpha_j-\beta_1-\gamma)
\frac{(t_3-t_1)^{\rho_1+\rho_2-\alpha_j-\beta_1}}
{(t_1-t_3)^{\rho_1+\rho_2-\beta_1-\gamma}}
\left(\frac{t_3-t_2}{t_1-t_2}\right)^{\rho_1+\rho_2-\alpha_j-\gamma}\\
&(\rho_1+\rho_2-\alpha_i-\beta_1)
\frac{\Gamma(-\gamma)\Gamma(\alpha_j+1)}
{\Gamma(\alpha_j-\rho_1)\Gamma(\alpha_j-\rho_2)}
\frac{\prod_{k\neq j}^n\Gamma(\alpha_j-\alpha_k)}
{\prod_{k=1}^{n-1}\Gamma(\alpha_j+\beta_k-\rho_1-\rho_2)}
\end{split}
\end{equation}

\begin{equation}
\begin{split}
(C_{32}^{(2n)})_j
&=(-1)^{n}
e(\tfrac{1}{2}(\rho_1+\rho_2-\alpha_1-\beta_j-\gamma))
\frac{(t_3-t_2)^{\rho_1+\rho_2-\alpha_1-\gamma}}
{(t_2-t_3)^{\rho_1+\rho_2-\alpha_1-\beta_j}}
\left(\frac{t_3-t_1}{t_2-t_1}\right)^{\rho_1+\rho_2-\beta_j-\gamma}\\
&(\rho_1+\rho_2-\alpha_1-\beta_j)
\Gamma(-\gamma)\Gamma(\beta_j+1)
\frac{\prod_{k\neq j}^{n-1}\Gamma(\beta_j-\beta_k)}
{\prod_{k=1}^n\Gamma(\alpha_k+\beta_j-\rho_1-\rho_2)}
\end{split}
\end{equation}

$({\rm III}^*)_{2n+1}$
\begin{equation}
\begin{split}
(C_{12}^{(2n+1)})_{ij}
&=(-1)^{n}
e(\tfrac{1}{2}(\rho_1+\rho_2-\alpha_i-\beta_j-\gamma)
\frac{(t_1-t_2)^{\rho_1+\rho_2-\alpha_i-\gamma}}
{(t_2-t_1)^{\rho_1+\rho_2-\beta_j-\gamma}}
\left(\frac{t_1-t_3}
{t_2-t_3}\right)^{\rho_1+\rho_2-\alpha_i-\beta_j}\\
&\frac{\Gamma(-\alpha_i)\Gamma(\beta_j+1)}{\Gamma(1+\rho_1-\alpha_i)\Gamma(\beta_j-\rho_1)}
\frac{\prod_{k\neq j}^n\Gamma(\beta_j-\beta_k)}
{\prod_{k\neq i}^n\Gamma(\beta_j+\alpha_{k}-\rho_1-\rho_2)}
\frac{\prod_{k\neq i}^n\Gamma(1+\alpha_k-\alpha_i)}
{\prod_{k\neq j}^n\Gamma(1+\rho_1+\rho_2-\alpha_i-\beta_k)}
\end{split}
\end{equation}

\begin{equation}
\begin{split}
(C_{13}^{(2n+1)})_{i}
&=(-1)^{n}e(\tfrac{1}{2}(\rho_1+\rho_2-\alpha_i-\beta_1-\gamma)
\frac{(t_1-t_3)^{\rho_1+\rho_2-\alpha_i-\beta_1}}
{(t_3-t_1)^{\rho_1+\rho_2-\beta_1-\gamma}}
\left(\frac{t_1-t_2}{t_3-t_2}\right)^{\rho_1+\rho_2-\alpha_i-\gamma}\\
&
\frac{\Gamma(\gamma+1)\Gamma(-\alpha_i)}{\Gamma(1+\rho_1-\alpha_i)}
\frac{\prod_{k\neq i}^n\Gamma(1+\alpha_k-\alpha_i)}{\prod_{k=1}^n\Gamma(1+\rho_1+\rho_2-\alpha_i-\beta_k)}
\end{split}
\end{equation}

\begin{equation}
\begin{split}
(C_{23}^{(2n+1)})_i
&=(-1)^{n}
e(\tfrac{1}{2}(\rho_1+\rho_2-\alpha_1-\beta_i-\gamma))
\frac{(t_2-t_3)^{\rho_1+\rho_2-\alpha_1-\beta_i}}
{(t_3-t_2)^{\rho_1+\rho_2-\alpha_1-\gamma}}
\left(\frac{t_2-t_1}{t_3-t_1}\right)^{\rho_1+\rho_2-\beta_i-\gamma}
\\
&
\frac{\Gamma(\gamma+1)\Gamma(-\beta_i)}
{\Gamma(1+\rho_1-\beta_i)}
\frac{\prod_{k\neq i}^n\Gamma(1+\beta_k-\beta_i)}
{\prod_{k=1}^n\Gamma(1+\rho_1+\rho_2-\alpha_k-\beta_i)}\\
\end{split}
\end{equation}

\begin{equation}
\begin{split}
(C_{21}^{(2n+1)})_{ij}
&=(-1)^{n}
e(\tfrac{-1}{2}(\rho_1+\rho_2-\alpha_j-\beta_i-\gamma)
\frac{(t_2-t_1)^{\beta_i+\gamma-\rho_1-\rho_2}}{(t_1-t_3)^{\rho_1+\rho_2-\alpha_i-\beta_j}}
\left(\frac{t_2-t_3}
{t_1-t2}\right)^{\rho_1+\rho_2-\alpha_j-\gamma}\\
&\frac{\Gamma(\alpha_j+1)\Gamma(-\beta_i)}{\Gamma(\alpha_j-\rho_1)\Gamma(1+\rho_1-\beta_i)}
\frac{\prod_{k\neq i}^n\Gamma(1+\beta_k-\beta_j)}
{\prod_{k\neq j}^n\Gamma(1+\rho_1+\rho_2-\beta_i-\alpha_{k})}
\frac{\prod_{k\neq j}^n\Gamma(\alpha_j-\alpha_k)}
{\prod_{k\neq i}^n\Gamma(\alpha_j+\beta_k-\rho_1-\rho_2)}
\end{split}
\end{equation}
\end{Theorem}

\begin{equation}
\begin{split}
(C_{31}^{(2n+1)})_{j}
&=(-1)^{n}e(\tfrac{1}{2}(\rho_1+\rho_2-\alpha_j-\beta_1-\gamma)
\frac{(t_3-t_1)^{\rho_1+\rho_2-\alpha_j-\beta_1}}{(t_1-t_3)^{\rho_1+\rho_2-\beta_1-\gamma}}
\left(\frac{t_3-t_2}{t_1-t_2}\right)^{\rho_1+\rho_2-\alpha_j-\gamma} \\
&
\frac{\Gamma(-\gamma)\Gamma(\alpha_j+1)}{\Gamma(\alpha_j-\rho_1)}
\frac{\prod_{k\neq j}^n\Gamma(\alpha_j-\alpha_k)}{\prod_{k=1}^n
\Gamma(\alpha_j+\beta_k-\rho_1-\rho_2)}
\end{split}
\end{equation}

\begin{equation}
\begin{split}
(C_{32}^{(2n+1)})_j
&=(-1)^{n}
e(\tfrac{1}{2}(\rho_1+\rho_2-\alpha_1-\beta_i-\gamma))
\frac{(t_3-t_2)^{\rho_1+\rho_2-\alpha_1-\gamma}}
{(t_2-t_3)^{\rho_1+\rho_2-\alpha_1-\beta_i}}
\left(\frac{t_3-t_1}{t_2-t_1}\right)^{\rho_1+\rho_2-\beta_i-\gamma}
\\
&
\frac{\Gamma(-\gamma)\Gamma(\beta_j+1)}
{\Gamma(\beta_j-\rho_1)}
\frac{\prod_{k\neq j}^n\Gamma(\beta_j-\beta_k)}
{\prod_{k=1}^n\Gamma(\alpha_k+\beta_j-\rho_1-\rho_2)}\\
\end{split}
\end{equation}
\begin{Theorem}
For the Okubo systems \eqref{eq:IV} of type  $({\rm IV})_6$, the connection coefficients of the canonical solution matrix are expressed as follows.
\begin{equation}
\begin{split}
&(C_{12})_{ij}=
\ds{e(\tfrac{1}{2}(-\alpha_i-\beta_j))
\frac{(t_2-t_1)^{\rho_3-\alpha_i}}{(t_1-t_2)^{\rho_3-\beta_j}}
\frac{\Gamma(-\alpha_{i})\Gamma(\beta_j+1)}{
\prod_{k=1}^3\Gamma(1+\rho_k-\alpha_i)}}
\\
&\ds{
\frac{\prod_{k\neq j}^2\Gamma(\beta_j-\beta_{k})
\prod_{k\neq i,4}^{4}\Gamma(1+\alpha_{k}-\alpha_i)}
{\prod_{k\neq i}^3  \Gamma(\beta_j+\alpha_{k}+\alpha_4-\rho_1-\rho_2-\rho_3)
\prod_{k\neq j}^2\Gamma(1+\rho_1+\rho_2+\rho_3-\alpha_i-\alpha_4-\beta_k)}
\quad (i=1,2,3;j=1,2).}
\end{split}
\end{equation}

\begin{equation}
\begin{split}
&(C_{12})_{4j}=\ds{\frac{e(\tfrac{1}{2}(-\alpha_4-\beta_j))}{\prod_{k\neq j}^2\prod_{k\neq 1}^3(\alpha_1+\alpha_k+\beta_j-\rho_1-\rho_2-\rho_3)}
\frac{(t_2-t_1)^{\rho_3-\alpha_4}}{(t_1-t_2)^{\rho_3-\beta_j}}
\frac{\Gamma(-\alpha_{4})\Gamma(\beta_j+1)}{
\prod_{k=1}^3\Gamma(1+\rho_k-\alpha_4)}}
\\
&
\ds{
\frac{\prod_{k\neq j}^2\Gamma(\beta_j-\beta_{k})
\prod_{k\neq 4}^{4}\Gamma(1+\alpha_{k}-\alpha_i)}
{\prod_{k\neq 1,4}^4  \Gamma(\beta_j+\alpha_{k}+\alpha_1-\rho_1-\rho_2-\rho_3)
\prod_{k\neq j}^2\Gamma(1+\rho_1+\rho_2+\rho_3-\alpha_1-\alpha_4-\beta_k)}}
\quad  (j=1,2).
\end{split}
\end{equation}

\begin{equation}
\begin{split}
&(C_{21})_{ij}
=
-\ds{e(\tfrac{1}{2}(\alpha_j+\alpha_4+\beta_i-2\rho_3))
\frac{(t_2-t_1)^{\alpha_4-\beta_i}}{(t_1-t_2)^{\alpha_4-\alpha_j}}
\frac{\Gamma(-\beta_{i})\Gamma(\alpha_j+1)}{\Gamma(\alpha_j-\rho_1)
\Gamma(\alpha_j-\rho_2)\Gamma(\alpha_j-\rho_3)}}\\
&
\ds{\frac{\prod_{k\neq j}^4\Gamma(\alpha_j-\alpha_{k})
\prod_{k\neq i}^2\Gamma(1+\beta_{k}-\beta_i)}
{\prod_{k\neq i}^2\Gamma(\alpha_j+\alpha_4+\beta_{k}-\rho_1-\rho_2-\rho_3)
\prod_{k\neq j,4}^4\Gamma(1+\rho_1+\rho_2+\rho_3-\alpha_k-\beta_i-\alpha_4)}}
\quad (i=1,2;j=1,2,3).
\end{split}
\end{equation}

\begin{equation}
\small
\begin{split}
(C_{21})_{i4}
&=
\ds{\prod_{k\neq 1}^3(\alpha_1+\alpha_k+\beta_j-\rho_1-\rho_2-\rho_3)
e(\tfrac{1}{2}(\alpha_1+\alpha_4+\beta_i-2\rho_3))
\frac{(t_2-t_1)^{\alpha_4-\beta_i}}{(t_1-t_2)^{\alpha_4-\alpha_j}}
\frac{\Gamma(-\beta_{i})\Gamma(\alpha_4+1)}{\Gamma(\alpha_4-\rho_1)
\Gamma(\alpha_4-\rho_2)\Gamma(\alpha_4-\rho_3)}}\\
&
\ds{\frac{\prod_{k\neq j}^4\Gamma(\alpha_j-\alpha_{k})
\prod_{k\neq i}^2\Gamma(1+\beta_{k}-\beta_i)}
{\prod_{k\neq i}^2\Gamma(\alpha_1+\alpha_4+\beta_{k}-\rho_1-\rho_2-\rho_3)
\prod_{k\neq 1,4}^4\Gamma(1+\rho_1+\rho_2+\rho_3-\alpha_k-\beta_i-\alpha_1)}}
\quad (i=1,2).
\end{split}
\end{equation}
\end{Theorem}

\begin{Theorem}
For the Okubo systems \eqref{eq:IV*} of types  $({\rm IV}^*)_6$, the connection coefficients of the canonical solution matrices are expressed as follows.
\begin{equation}
\begin{split}
(C_{12})_{ij}&=
e(\tfrac{1}{2}(2\rho_1+\rho_2-\alpha_i-\beta_j-\gamma_1-\gamma_2))
\frac{(t_1-t_2)^{\rho_1+\rho_2-\alpha_i-\gamma_1}}
{(t_2-t_1)^{\rho_1+\rho_2-\beta_j-\gamma_1}}
\left(\frac{t_1-t_3}{t_2-t_3}\right)^{\rho_1+\rho_2-\alpha_i-\beta_j}\\
&\frac{\Gamma(-\alpha_i)\Gamma(\beta_j+1)}
{\Gamma(1+\rho_1-\alpha_i)\Gamma(\beta_j-\rho_1)}
\frac{\prod_{k\neq j}^2\Gamma(\beta_j-\beta_{k})
\prod_{k\neq i}^2\Gamma(1+\alpha_k-\alpha_i)}
{\prod_{k\neq i}^2\Gamma(\alpha_k+\beta_j+\gamma_2-2\rho_1-\rho_2)
\prod_{k\neq j}^2\Gamma(1+2\rho_1+\rho_2-\alpha_i-\beta_{k}-\gamma_2)}
\end{split}
\end{equation}

\begin{equation}
\begin{split}
(C_{13})_{ij}
&=(\prod_{k\neq i}^2h_{k2j}\prod_{k\neq j}^2h_{i2k})^{\delta_{2j}}e(\tfrac{1}{2}(2\rho_1+\rho_2-\alpha_i-\beta_1-\gamma_1-\gamma_2)
\frac{(t_1-t_3)^{\rho_1+\rho_2-\alpha_i-\beta_1}}
{(t_3-t_1)^{\rho_1+\rho_2-\beta_1-\gamma_j}}
\left(\frac{t_1-t_2}{t_3-t_2}\right)^{\rho_1+\rho_2-\alpha_i-\gamma_j}\\
&
\frac{\Gamma(\gamma_j+1)\Gamma(-\alpha_i)}{\Gamma(1+\rho_1-\alpha_i)\Gamma(\gamma_j-\rho_1)}
\frac{\prod_{k\neq j}^2\Gamma(\gamma_j-\gamma_k)
\prod_{k\neq i}^2\Gamma(1+\alpha_k-\alpha_i)}{
\prod_{k\neq i}^2\Gamma(1+\alpha_k+\beta_2+\gamma_j-2\rho_1-\rho_2)
\prod_{k\neq j}^2\Gamma(1+2\rho_1+\rho_2-\alpha_i-\beta_2-\gamma_k)}
\end{split}
\end{equation}

\begin{equation}
\begin{split}
(C_{23})_{ij}
&=(\prod_{k\neq i}^2h_{2kj}\prod_{k\neq j}^2h_{2ik})^{\delta_{2j}}
e(\tfrac{1}{2}(\rho_1+\rho_2-\alpha_1-\beta_i-\gamma_j))
\frac{(t_2-t_3)^{\rho_1+\rho_2-\alpha_1-\beta_i}}
{(t_3-t_2)^{\rho_1+\rho_2-\alpha_1-\gamma_j}}
\left(\frac{t_2-t_1}{t_3-t_1}\right)^{\rho_1+\rho_2-\beta_i-\gamma_j}
\\
&
\frac{\Gamma(\gamma_j+1)\Gamma(-\beta_i)}
{\Gamma(1+\rho_1-\beta_i)\Gamma(\gamma_j-\rho_1)}
\frac{\prod_{k\neq j}^2\Gamma(\gamma_j-\gamma_k)
\prod_{k\neq i}^2\Gamma(1+\beta_k-\beta_i)}
{\prod_{k\neq i}^2\Gamma(1+\alpha_2+\beta_k+\gamma_j-2\rho_1-\rho_2)
\prod_{k\neq j}^2\Gamma(1+2\rho_1+\rho_2-\alpha_2-\beta_i-\gamma_k)}\\
\end{split}
\end{equation}

\begin{equation}
\begin{split}
(C_{21})_{ij}&=
e(\tfrac{1}{2}(2\rho_1+\rho_2-\alpha_i-\beta_j-\gamma_1-\gamma_2))
\frac{(t_1-t_2)^{\rho_1+\rho_2-\alpha_j-\gamma_1}}
{(t_2-t_1)^{\rho_1+\rho_2-\beta_i-\gamma_1}}
\left(\frac{t_1-t_3}{t_2-t_3}\right)^{\rho_1+\rho_2-\alpha_i-\beta_j}\\
&\frac{\Gamma(\alpha_j+1)\Gamma(-\beta_i)}
{\Gamma(1+\rho_1-\beta_i)\Gamma(\alpha_j-\rho_1)}
\frac{\prod_{k\neq j}^2\Gamma(\alpha_j-\alpha_k)
\prod_{k\neq i}^2\Gamma(1+\beta_{k}-\beta_i)}
{\prod_{k\neq i}^2\Gamma(\alpha_j+\beta_k+\gamma_2-2\rho_1-\rho_2)
\prod_{k\neq j}\Gamma(1+2\rho_1+\rho_2-\alpha_k-\beta_{i}-\gamma_2)}
\end{split}
\end{equation}

\begin{equation}
\begin{split}
(C_{31})_{ij}&=\frac{e(\tfrac{1}{2}(2\rho_1+\rho_2-\alpha_j-\beta_1-\gamma_i-\gamma_2))}{(\prod_{k\neq i}^2h_{j2k}\prod_{k\neq j}^2h_{k2i})^{\delta_{i2}}}
\frac{(t_1-t_3)^{\rho_1+\rho_2-\alpha_j-\beta_1}}
{(t_3-t_1)^{\beta_1+\gamma_i-\rho_1-\rho_2}}
\left(\frac{t_1-t_2}{t_3-t_2}\right)^{\rho_1+\rho_2-\alpha_j-\gamma_i}\\
&\frac{\Gamma(-\gamma_i)\Gamma(1+\alpha_j)}
{\Gamma(1+\rho_1-\gamma_i)\Gamma(\alpha_j-\rho_1)}
\frac{\prod_{k\neq j}\Gamma(\alpha_j-\alpha_k)
\prod_{k\neq i}\Gamma(1+\gamma_i-\gamma_k)}
{\prod_{k\neq i}\Gamma(\alpha_j+\beta_2+\gamma_k-2\rho_1-\rho_2)
\prod_{k\neq j}^2\Gamma(2\rho_1+\rho_2-\alpha_k-\beta_{2}-\gamma_i)
}
\end{split}
\end{equation}

\begin{equation}
\begin{split}
(C_{32})_{ij}&=\frac{e(\tfrac{1}{2}(2\rho_1+\rho_2-\alpha_1-\beta_j-\gamma_1-\gamma_2))}{(\prod_{k\neq i}^2h_{2jk}\prod_{k\neq i}^2h_{2ki})^{\delta_{i2}}}
\frac{(t_2-t_3)^{\rho_1+\rho_2-\alpha_1-\beta_j}}
{(t_3-t_2)^{\alpha_1+\gamma_i-\rho_1-\rho_2}}
\left(\frac{t_2-t_1}{t_3-t_1}\right)^{\rho_1+\rho_2-\beta_j-\gamma_i}\\
&\frac{\Gamma(-\gamma_i)\Gamma(1+\beta_j)}
{\Gamma(1+\rho_1-\gamma_i)\Gamma(\beta_j-\rho_1)}
\frac{\prod_{k\neq j}\Gamma(\beta_j-\beta_{k})
\prod_{k\neq i}\Gamma(1+\gamma_i-\gamma_k)}
{\prod_{k\neq i}\Gamma(\alpha_2+\beta_j+\gamma_k-2\rho_1-\rho_2)
\prod_{k\neq j}\Gamma(\rho_1+\rho_2-\alpha_2-\beta_k-\gamma_i)}\\
\end{split}
\end{equation}
\end{Theorem}

\subsection{Proof of the main theorems}

Before the proof of Theorems 4.1, 4.2 and 4.3, we give a simple remark concerning the symmetry of Okubo systems.
\begin{Lemma}\label{lem:sym}
Assume that the Okubo system \eqref{eq:Okubo}  is obtained by conjugation from an Okubo system
\begin{equation}\label{eq:OkuboB}
(x-T)\frac{d}{dx}Y=BY
\end{equation} as
\begin{equation}
A={\rm Ad}(D)B=DBD^{-1},\quad D={\rm diag}(D_1,\ldots,D_r),\quad D_i\in {\rm GL}(n_i,\mathbb{C})
\quad (i=1,\ldots,r).
\end{equation}
Let $C_{ij}^A$ and $C_{ij}^B$ $(1\leq i,j \leq r)$ be the connection matrices for the canonical solution matrices of \eqref{eq:Okubo} and \eqref{eq:OkuboB}, respectively.
Then we have 
\begin{equation}
C_{ij}^{A}=D_iC_{ij}^{B}D_j^{-1}\quad (i,j=1,\ldots,r).
\end{equation}
\end{Lemma}
In the context of the Okubo system in Yokoyama's list, we apply this lemma to analyze how the connection coefficients transform under the permutation of  characteristic exponents.

\par\medskip
\noindent
Case ${\rm II}^*,{\rm III}^*$: 
Recall that the canonical form of the system $({\rm III}^*)_3$ is constructed by the middle convolution ${\rm mc}_{\mu}$ for the 
differential system 
\begin{equation}
\frac{d}{dx}Y=\left(\frac{\alpha}{x-t_1}+\frac{\beta}{x-t_2}+\frac{\gamma}{x-t_3}\right)Y,
\quad Y=(x-t_1)^{\alpha}(x-t_2)^{\beta}(x-t_3)^{\gamma}.
\end{equation}
Since the Okubo system of type $({\rm III}^*)_3$ is equivalent to the Okubo system
of type $({\rm I}^*)_3$ up to conjugation, the connection coefficients in this case are directly 
computed as in \cite{SK}. 
The connection coefficients
$(C_{12}^{(3)})_{11}$ and $(C_{13}^{(3)})_1$ are in fact given as
\begin{equation}
\begin{split}
(C_{12}^{(3)})_{11}
=&-
e(\tfrac{1}{2}(\rho_1+\rho_2-\alpha_1-\beta_1-\gamma))
\frac{(t_1-t_2)^{\rho_1+\rho_2-\alpha_1-\gamma}}
{(t_2-t_1)^{\beta_1+\gamma-\rho_1-\rho_2}}\\
&\left(\frac{t_1-t_3}
{t_2-t_3}\right)^{\rho_1+\rho_2-\alpha_1-\beta_1}
\frac{\Gamma(-\alpha_1)\Gamma(\beta_1+1)}{\Gamma(1+\rho_1-\alpha_1)\Gamma(\beta_1-\rho_1)}
\end{split}
\end{equation}

\begin{equation}
\begin{split}
(C_{13}^{(3)})_{1}
=&-(\gamma-\rho_1)^{-1}
e(\tfrac{1}{2}(\rho_1+\rho_2-\alpha_1-\beta_1-\gamma))
\frac{(t_1-t_3)^{\rho_1+\rho_2-\alpha_1-\beta_1}}
{(t_3-t_1)^{\beta_1+\gamma-\rho_1-\rho_2}}\\
&\left(\frac{t_1-t_2}{t_3-t_2}\right)^{\rho_1+\rho_2-\alpha_1-\gamma}
\frac{\Gamma(-\alpha_i)\Gamma(\gamma+1)}{\Gamma(1+\rho_1-\alpha_1)\Gamma(\gamma-\rho_1)}
\end{split}
\end{equation}

Using Theorem \ref{conn:mcadd}, the connection coefficients $(C_{12}^{(2n+1)})_{11}$ and 
$(C_{13}^{(2n+1)})_{1}$ of $({\rm III}^*)_{2n+1}$ are computed as follows:
\begin{equation}
\small
\begin{split}
(C_{12}^{(2n+1)})_{11}&=(-1)^{n-1}
\prod_{k=2}^{n}e(\tfrac{1}{2}(\beta_k^{(2k+1)}-\rho_1^{(2k+1)}))
(t_1-t_2)^{\beta_{k}^{(2k+1)}-\rho_1^{(2k+1)}}
e(\tfrac{1}{2}(\alpha_k^{(2k)}-\rho_1^{(2k)}))(t_2-t_1)^{\rho_1^{(2k)}-\alpha_k^{(2k)}}\\
&\prod_{k=2}^{n}
\frac{\Gamma(-\alpha_1^{(2k+1)})}{\Gamma(-\alpha_1^{(2k)})}
\frac{\Gamma(\beta_1^{(2k+1)}-\beta_k^{(2k+1)})}{\Gamma(\beta_1^{(2k+1)}-\rho_1^{(2k+1)})}
\frac{\Gamma(1+\alpha_k^{(2k)}-\alpha_1^{(2k)})}{\Gamma(1+\rho_1^{(2k)}-\alpha_1^{(2k)})}
\frac{\Gamma(\beta_1^{(2k)}+1)}{\Gamma(\beta_1^{(2k-1)}+1)}(C_{12}^{(3)})_{11}\\
&=(-1)^{n-1}e(\tfrac{1}{2}(-\rho_1^{(2n+1)}+\beta_2^{(5)}))e(\tfrac{1}{2}(-\rho_1^{(2n)}+\alpha_2^{(4)}))
(t_1-t_2)^{-\rho_1^{(2n+1)}+\beta_2^{(5)}}(t_2-t_1)^{\rho_1^{(2n)}-\alpha_2^{(4)}}
\\
&\frac{\Gamma(-\alpha_1^{(2n+1)})\Gamma(\beta_1^{(2n)}+1)}{\Gamma(-\alpha_1^{(3)})\Gamma(\beta_1^{(3)}+1)}
\prod_{k=2}^{n}\frac{\Gamma(\beta_1^{(2k+1)}-\beta_k^{(2k+1)})}{\Gamma(\beta_1^{(2k+1)}-\rho_1^{(2k+1)})}
\frac{\Gamma(1+\alpha_k^{(2k)}-\alpha_1^{(2k)})}{\Gamma(1+\rho_1^{(2k)}-\alpha_1^{(2k)})}(C_{12}^{(3)})_{11}\\
&=(-1)^{n}
e(\tfrac{1}{2}(-\rho_1^{(2n+1)}-\rho_1^{(2n)}+\beta_2^{(5)}+\alpha_2^{(4)}-\rho_1^{(3)}))
\frac{(t_1-t_2)^{-\rho_1^{(2n+1)}+\beta_1^{(3)}}}
{(t_2-t_1)^{\rho_1^{(2n)}-\alpha_2^{(4)}+\rho_1^{(3)}-\alpha_1^{(3)}}}
\\
&\left(\frac{t_1-t_3}
{t_2-t_3}\right)^{\rho_1^{(2n+1)}+\rho_2^{(2n+1)}-\alpha_1^{(2n+1)}-\beta_1^{(2n+1)}}
\frac{\Gamma(-\alpha_1^{(2n+1)})\Gamma(\beta_1^{(2n+1)}+1)}{\Gamma(1+\rho_1^{(2n+1)}-\alpha_1^{(2n+1)})\Gamma(\beta_1^{(2n+1)}-\rho_1^{(2n+1)})}\\
&\prod_{k=2}^{n}\frac{\Gamma(\beta_1^{(2n+1)}-\beta_k^{(2n+1)})}{\Gamma(\beta_1^{(2n+1)}+\alpha_{k}^{(2n+1)}-\rho_1^{(2n+1)}-\rho_2^{(2n+1)})}
\frac{\Gamma(1+\alpha_k^{(2n+1)}-\alpha_1^{(2n+1)})}					{\Gamma(1+\rho_1^{(2n+1)}+\rho_2^{(2n+1)}-\alpha_1^{(2n+1)}-\beta_k^{(2n+1)})}\\
&=(-1)^{n}
e(\tfrac{1}{2}(\rho_1+\rho_2-\alpha_1-\beta_1-\gamma)
\frac{(t_1-t_2)^{\rho_1+\rho_2-\alpha_1-\gamma}}
{(t_2-t_1)^{\beta_1+\gamma-\rho_1-\rho_2}}
\left(\frac{t_1-t_3}
{t_2-t_3}\right)^{\rho_1+\rho_2-\alpha_1-\beta_1}\\
&
\frac{\Gamma(-\alpha_1)\Gamma(\beta_1+1)}{\Gamma(1+\rho_1-\alpha_1)\Gamma(\beta_1-\rho_1)}
\prod_{k=2}^{n}\frac{\Gamma(\beta_1-\beta_k)}{\Gamma(\beta_1+\alpha_{k}-\rho_1-\rho_2)}
\frac{\Gamma(1+\alpha_k-\alpha_1)}{\Gamma(1+\rho_1+\rho_2-\alpha_1-\beta_k)}
\end{split}
\end{equation}

\begin{equation}
\small
\begin{split}
(C_{13}^{(2n+1)})_{1}&=(-1)^{n-1}\prod_{k=2}^n
\left(\frac{t_1-t_2}{t_3-t_2}\right)^{\beta_k^{(2k+1)}-\rho_1^{(2k+1)}}
\frac{e(\frac{1}{2}(\beta_k^{(2k+1)}-\rho_1^{(2k+1)}))\Gamma(-\alpha_1^{(2k+1)})}{\Gamma(-\alpha_1^{(2k)})}
\frac{\Gamma(\gamma^{(2k+1)}+1)}{\Gamma(\gamma^{(2k)}+1)}\\
&
e(\tfrac{1}{2}(\alpha_k^{(2k)}-\rho_1^{(2k)}))(t_3-t_1)^{\rho_1^{(2k)}-\alpha_k^{(2k)}}
\frac{\Gamma(1+\alpha_k^{(2k)}-\alpha_1^{(2k)})}{\Gamma(1+\rho_1^{(2k)}-\alpha_1^{(2k)})}
\frac{\Gamma(\gamma^{(2k)}+1)}{\Gamma(\gamma^{(2k-1)}+1)}(C_{13}^{(3)})_{11}\\
&=(-1)^{n-1}e(\frac{1}{2}(-\rho_1^{(2n+1)}-\rho_2^{(2n+1)}+\beta_2^{(5)}+\alpha_2^{(2)}))
(t_3-t_1)^{\rho_2^{(2n+1)}-\alpha_2^{(4)}}
\left(\frac{t_1-t_2}{t_3-t_2}\right)^{-\rho_1^{(2n+1)}+\beta_2^{(5)}}\\
&\frac{\Gamma(\gamma^{(2n+1)}+1)}{\Gamma(\gamma^{(3)}+1)}\frac{\Gamma(-\alpha_1^{(2n+1)})}{\Gamma(-\alpha_1^{(4)})}
\prod_{k=2}^n\frac{\Gamma(1+\alpha_k^{(2n+1)}-\alpha_1^{(2n+1)})}{\Gamma(1+\rho_1^{(2n+1)}+\rho_2^{(2n+1)}-\alpha_1^{(2n+1)}-\beta_k^{(2n+1)})}(C_{13}^{(3)})_{11}\\
&=(-1)^{n}e(\tfrac{1}{2}(\rho_1+\rho_2-\alpha_1-\beta_1-\gamma))
(t_1-t_3)^{\rho_1+\rho_2-\alpha_1-\beta_1}(t_3-t_1)^{\beta_1+\gamma-\rho_1-\rho_2}\\
&\left(\frac{t_1-t_2}{t_3-t_2}\right)^{\rho_1+\rho_2-\alpha_1-\gamma}
\frac{\Gamma(\gamma+1)\Gamma(-\alpha_1)}{\Gamma(1+\rho_1-\alpha_1)}
\frac{\prod_{k=2}^n\Gamma(1+\alpha_k-\alpha_1)}{\prod_{k=1}^n\Gamma(1+\rho_1+\rho_2-\alpha_1-\beta_k)}
\end{split}
\end{equation}
The connection coefficients $(C_{12})_{ij}$ and $(C_{13})_i$ for other $i,j$ are derived by Lemma {\ref{lem:sym}. In the following we denote by $\sigma_{ij}^{\alpha}$ (resp. $\sigma_{ij}^\beta$) 
 the operation
which exchanges the parameters $\alpha_i$ and $\alpha_j$ (resp. $\beta_i$ and $\beta_j$).  
Then the matrix $A$ of the canonical form of type $({\rm III}^*)_{2n+1}$ satisfies
\begin{equation}\label{Asym}
A={\rm Ad}(D)\sigma_{1i}^{\alpha}\sigma_{1j}^{\beta}(A),\quad 
D={\rm diag}(S_{1i},S_{1j},1)
\end{equation}
where  $S_{ij}$ is the permutation matrix corresponding the transposition $(ij)$ of matrix indices. From \eqref{Asym},  applying Lemma \ref{lem:sym} to $B=\sigma_{1i}^{\alpha}\sigma_{1j}^{\beta}(A)$ we obtain
\begin{equation}
C_{12}=S_{1i}\sigma_{1i}^{\alpha}\sigma_{1j}^{\beta}(C_{12})S_{1j},\quad 
C_{13}=S_{1i}\sigma_{1i}^{\alpha}\sigma_{1j}^{\beta}(C_{13}).
\end{equation}
Therefore the 
connection coefficients $(C_{12})_{ij}$ and $(C_{13})_i$ are given by 
\begin{equation}
(C_{12})_{ij}=\sigma_{1i}^{\alpha}\sigma_{1j}^{\beta}(C_{12})_{11},
\quad 
(C_{13})_i=\sigma_{1i}^{\alpha}(C_{13})_1.
\end{equation}
The other connection coefficients of type ${\rm III}^*$ and connection coefficients of type ${\rm II}^*$ are computed similarly.\\

\par\medskip
\noindent {\bf Case} ${\rm IV}$: 
The canonical forms of type ${\rm IV}$ is given by the operation
\begin{equation}
({\rm IV})_6={\rm add}_{(\rho_3,0)}\circ {\rm mc}_{-\rho_3-c}\circ {\rm add}_{(c,0)}({\rm III}_5).
\end{equation}
Using Theorem \ref{conn:mcadd}, we can completely determine $(C_{12}^{(6)})_{11}$ 
for \eqref{eq:IV} by the following computation:
\begin{equation}
\small
\begin{split}
(C_{12}^{(6)})_{11}&=(C_{(11)2}^{mc})_{11}=
-e(\tfrac{1}{2}(\rho_3+c))(t_2-t_1)^{-\rho_3-c}
\frac{\Gamma(1+\rho_3-\alpha_1)}{\Gamma(1-\alpha_1-c)}(C_{12}^{(5)})_{11}
\frac{\Gamma(\beta_1-\rho_3-c+1)}{\Gamma(\beta_1+1)}\\
&=-e(\tfrac{1}{2}(\alpha_4-\rho_3))(t_2-t_1)^{\rho_3-\alpha_4}
\frac{\Gamma(1+\alpha_4-\alpha_1)}{\Gamma(1+\rho_3-\alpha_1)}(C_{12})_{11}
\frac{\Gamma(\beta_1^{(6)}+1)}{\Gamma(\beta_1^{(5)}+1)}\\
&=
e(\tfrac{1}{2}(\alpha_4-\rho_3))(t_2-t_1)^{\rho_3-\alpha_4}
\frac{\Gamma(1+\alpha_4-\alpha_1)}{\Gamma(1+\rho_3-\alpha_1)}
\frac{\Gamma(\beta_1^{(6)}+1)}{\Gamma(\beta_1^{(5)}+1)}
\frac{(t_2-t_1)^{\alpha_4-\alpha_i}}{(t_1-t_2)^{\rho_3-\beta_j}}
e(\tfrac{1}{2}(\rho_3-\alpha_1-\beta_1-\alpha_4))\\
&\frac{\Gamma(-\alpha_{1})\Gamma(\beta_1^{(5)}+1)}{\Gamma(1+\rho_1-\alpha_1)\Gamma(1+\rho_2-\alpha_1)}
\frac{\prod_{k\neq 1}^2\Gamma(\beta_1-\beta_{k})}{\prod_{k\neq 1}^{3}\Gamma(\beta_1+\alpha_{k}+\alpha_4-\rho_1-\rho_2-\rho_3)}
\frac{\prod_{k\neq 1}^{3}\Gamma(1+\alpha_{k}-\alpha_1)}{\prod_{k\neq 1}^2\Gamma(1+\rho_1+\rho_2+\rho_3-\beta_k-\alpha_1-\alpha_4)}\\
&=e(\tfrac{1}{2}(-\alpha_1-\beta_1))
\frac{(t_2-t_1)^{\rho_3-\alpha_1}}{(t_1-t_2)^{\rho_3-\beta_j}}
\frac{\Gamma(-\alpha_{1})\Gamma(\beta_1+1)}{\Gamma(1+\rho_1-\alpha_1)\Gamma(1+\rho_2-\alpha_1)\Gamma(1+\rho_3-\alpha_1)}
\\
&
\frac{\prod_{k\neq 1}^2\Gamma(\beta_1-\beta_{k})}{\prod_{k\neq 1}^{3}\Gamma(\beta_1+\alpha_{k}+\alpha_4-\rho_1-\rho_2-\rho_3)}
\frac{\prod_{k\neq 1}^{4}\Gamma(1+\alpha_{k}-\alpha_1)}{\prod_{k\neq 1}^2\Gamma(1+\rho_1+\rho_2+\rho_3-\beta_k-\alpha_1-\alpha_4)}.
\end{split}
\end{equation}
As in the cases of types ${\rm II}^*$ and ${\rm III}^*$, we determine the other connection coefficients  by Lemma \ref{lem:sym} and operations $\sigma_{ij}^{\alpha}$, $\sigma_{ij}^{\beta}$.
For $\sigma_{1i}^{\alpha}$, $\sigma_{12}^{\beta}$, the Okubo system \eqref{eq:IV} satisfies the relations
\begin{equation}\label{relationIV1}
A={\rm Ad}(D)\sigma_{1i}^{\alpha}\sigma_{12}^{\beta}(A)\quad D={\rm diag}(S_{1i},S_{12}),\quad (i=1,2,3).
\end{equation}
By taking account of Lemma \ref{lem:sym} and \eqref{relationIV1}, the connection coefficients $(C_{12})_{ij}$ $(i=1,2,3;j=1,2)$
 are given by 
\begin{equation}
(C_{12})_{ij}=\sigma_{1i}^{\alpha}\sigma_{1j}^{\beta}(C_{12})_{11}.
\end{equation}
Although the system \eqref{eq:IV} is {\em not} symmetric with respect to ($\alpha_1,\alpha_4$), for   the operation $\sigma_{14}^{\alpha}$
it satisfies the relation
\begin{equation}
A={\rm Ad}(D_{14}^{\alpha})\sigma_{14}^{\alpha}(A)
\end{equation}
where $D_{14}^{\alpha}={\rm diag}(d_1,\ldots,d_6){\rm diag}(S_{14},\,I_2)$ is the matrix defined by
\begin{equation}
\begin{split}
&d_1=\prod_{k=1}^2
(\alpha_1+\alpha_2+\beta_k-\rho_1-\rho_2-\rho_3)(\alpha_1+\alpha_3+\beta_k-\rho_1-\rho_2-\rho_3),\\
&d_2=-\prod_{k=1}^2(\alpha_1+\alpha_3+\beta_k-\rho_1-\rho_2-\rho_3),
\quad 
d_3=-\prod_{k=1}^2(\alpha_1+\alpha_2+\beta_k-\rho_1-\rho_2-\rho_3),\quad
d_4=1,\\
&d_5=\prod_{k\neq 1}^3
(\alpha_1+\alpha_k+\beta_1-\rho_1-\rho_2-\rho_3),
\quad
d_6=\prod_{k\neq 1}^3
(\alpha_1+\alpha_k+\beta_2-\rho_1-\rho_2-\rho_3).
\end{split}
\end{equation}
Therefore the connection coefficient $(C_{12})_{41}$ is derived by
\begin{equation}
(C_{12})_{41}=d_{4}\sigma_{14}^{\alpha}(C_{12})_{11}d_6^{-1}.
\end{equation}
Since $C_{12}=\sigma_{12}^{\beta}(C_{12})S_{12}$, we obtain
\begin{equation}
(C_{12})_{4j}=\sigma_{12}^{\beta}(d_{4}\sigma_{14}^{\alpha}(C_{12})_{11}d_6^{-1}).
\end{equation} 
The other connection coefficients $(C_{21})_{ij}$ of type ${\rm IV}$ are computed similarly.

\par\medskip
\noindent
Case ${\rm IV}^*$: 
The canonical form of type ${\rm IV}^*$ is constructed by
\begin{equation}
({\rm IV}^*)_6={\rm add}_{(0,0,\rho_1)}\circ{\rm mc}_{-c-\rho_1}\circ{\rm add}_{(0,0,c)}({\rm III}^*)_5 .
\end{equation}
The connection coefficients $(C_{12})_{11},(C_{13})_{11},(C_{23})_{11}$  are directly computed
from the connection coefficients $(C_{ij})_{11}^{(5)}$ of type $({\rm III}^*)_5$ by Theorem \ref{conn:mcadd}.
The other connection coefficients can be determined by combining Lemma \ref{lem:sym} and 
 the operations $\sigma_{12}^{\alpha}$, $\sigma_{12}^{\beta}$, $\sigma_{12}^{\gamma}$.
For $\sigma_{12}^{\alpha}$, $\sigma_{12}^{\beta}$, the Okubo system \eqref{eq:IV*} satisfies the relation
\begin{equation}\label{symIV*}
\begin{split}
A={\rm Ad}(D)\sigma_{12}^{\alpha}\sigma_{12}^{\beta}(A),\quad D={\rm diag}(S_{12},S_{12},I_2).
\\
\end{split}
\end{equation}
For $\sigma_{12}^{\gamma}$, the Okubo system \eqref{eq:IV*} satisfies
\begin{equation}
A={\rm Ad}(D^{\gamma})\sigma_{12}^{\gamma}(A),
\end{equation}
where $D^{\gamma}$ is the matrix  given by
\begin{equation}
D^{\gamma}={\rm diag}(h_{111}h_{121},\, h_{211}h_{221},\, -h_{111}h_{211},\, -h_{121}h_{221},\, 
 h_{111}h_{121}h_{211}h_{221},\, 1){\rm diag}(I_2,\,I_2,\,S_{12}).
\end{equation}
With these relations, the other connection matrices are determined by Lemma \ref{lem:sym} 
as follows:
\begin{equation}
\begin{split}
C_{12}&=S_{1i}\sigma_{1i}^{\alpha}\sigma_{1j}^{\beta}(C_{12})S_{1j},\quad 
C_{13}=S_{1i}\sigma_{1i}^{\alpha}(C_{13}),\quad 
C_{23}=S_{1i}\sigma_{1i}^{\beta}(C_{23}),\\
C_{13}&=
\begin{pmatrix}
h_{111}h_{121}&0\\
0&h_{211}h_{221}
\end{pmatrix}
\sigma_{12}^{\gamma}(C_{13})
S_{12}
\begin{pmatrix}
h_{111}h_{121}h_{211}h_{221}&0\\
0&1
\end{pmatrix}^{-1},\\
C_{23}&=
\begin{pmatrix}
-h_{111}h_{211}&0\\
0&-h_{121}h_{221}
\end{pmatrix}
\sigma_{12}^{\gamma}(C_{23})
S_{12}
\begin{pmatrix}
h_{111}h_{121}h_{211}h_{221}&0\\
0&1
\end{pmatrix}^{-1}.
\end{split}
\end{equation}
This completes the proof of the main theorems.


\end{document}